\documentclass[12pt]{article}
\usepackage{color}

\usepackage{amsmath}
\usepackage{amsfonts}
\usepackage{amsthm}
\usepackage{amssymb}
\usepackage{bm}
\usepackage[applemac]{inputenc}
\usepackage[mathcal,mathbf]{euler}
\usepackage{enumerate}
\makeatletter

\newtheoremstyle{mystyle}{}{}{\slshape}{2pt}{\scshape}{.}{ }{} 

\newtheorem{thm}{Theorem}[section]
\newtheorem{cor}[thm]{Corollary}

\newtheorem{prop}[thm]{Proposition}
\newtheorem{lemme}[thm]{Lemma}

\newtheorem{question}[thm]{Question}

\theoremstyle{definition}
\newtheorem{defi}[thm]{Definition}
\theoremstyle{mystyle}

\theoremstyle{remark}
\newtheorem{rem}[thm]{Remark}

\newcommand{\llg}{\langle}
\newcommand{\rrg}{\rangle}

\DeclareMathOperator{\tp}{tp}
\DeclareMathOperator{\val}{val}
\DeclareMathOperator{\st}{st}

\title{Finding generically stable measures}
\author{Pierre Simon}

\begin{document}
\maketitle

\begin{abstract}
This work builds on \cite{NIP2} and \cite{NIP3} where Keisler measures over $NIP$ theories are studied. We discuss two constructions for obtaining generically stable measures in this context. First, we show how to symmetrize an arbitrary invariant measure to obtain a generically stable one from it.  Next, we show that suitable sigma-additive probability measures give rise to generically stable Keisler measures. Also included is a proof that generically stable measures over o-minimal theories and the p-adics are smooth.
\end{abstract}

\section*{Introduction}

A major theme in the study of $NIP$ theories is the investigation of stable phenomena in them. Generically stable types are defined in \cite{NIP2}: they are types that behave, as far as non-forking extensions are concerned, like stable types. This notion was generalized to Keisler measures in \cite{NIP3}. Keisler measures (introduced first in \cite{Keisler} and studied again with a different approach in \cite{NIP1}, \cite{NIP2} and \cite{NIP3}) are regular Borel probability measures on the space of types. Equivalently, they are finitely additive probability measures on the boolean algebra of definable sets. They can also be considered as types in continuous logic (see \cite{BY2}). Some basic facts about them are recalled in the first section.

After Keisler's seminal work \cite{Keisler}, measures were introduced again for the study of $NIP$ theories in Hrushovski, Peterzil and Pillay's paper \cite{NIP1}, focussing mainly on invariant measures of groups. The notion of \emph{fsg} group is introduced. Measures are studied more at depth in \cite{NIP2} in order to prove \emph{compact domination} for definably compact groups in o-minimal theories. Again, applications use only invariant measures on groups, but a general study is initiated. Also, the notion of generically stable type is defined and the question is asked of a generalization to measures. This is done in \cite{NIP3} by Hrushovski, Pillay and the author where generically stable measures are defined. A number of equivalent properties are given.

The following examples of generically stable measures are known:

- a generically stable type: this is the motivating example,

- the (translation) invariant measure of an $fsg$ group,

- the $A$-invariant measure extending an $fsg$ type over $A$ (\cite{NIP3} section 4),

- the average measure of an indiscernible segment (\cite{NIP3}, 3.7),

- the Keisler measure induced by a $\sigma$-additive measure on the standard model, $\mathbb R$ or $\mathbb Q_p$ (\cite{NIP3}, section 6). In fact, those measures are proven to be \emph{smooth}, a stronger property.

In this paper, we generalize the last two constructions in this list. First, we show how to \emph{symmetrize} any measure, or equivalently to average an indiscernible segment of measures. Second, we show that any $\sigma$-additive measure induces a generically stable measure, under the assumption that externally definable sets in two variables are measurable (so that Fubini applies). Smoothness of the induced measure is not true in general, and to recover the full result of \cite{NIP3} we also include a proof that all generically stable measures in $\mathbb R$ or $\mathbb Q_p$ are smooth. In order to prove those results, we first establish in Section 2 a criterion to recognize a product measure $\mu_x \otimes \lambda_y$ when $\mu_x$ happens to be finitely satisfiable over some small model. Our strategy for proving that a given measure $\mu$ is generically stable will then be to construct a symmetric measure $\eta_{x_1x_2...}$ in $\omega$ variables, and show using this criterion that it is the Morley sequence of $\mu$.
\\

We describe the canonical example of a generically stable measure we have in mind. Consider the following theory: the signature is $\{\leq, E\}$. The reduct to $\{\leq\}$ is a dense linear order, and the theory says that $E$ is an equivalence relation with infinitely many classes, each of which is dense with respect to $\leq$. This theory has no generically stable type over the main sort (because every type falls in some cut of the linear order). However, one can build a generically stable measure over the main sort by averaging types that fall in different cuts. More precisely, assume for example that we work over a model $(M,\leq,E)$ and we have an increasing embedding $f: ((0,1),\leq) \rightarrow (M,\leq)$, where $(0,1)$ denotes the standard unit open interval. Let $\lambda_0$ denote the standard Lebesgue measure on $(0,1)$. Define a Keisler measure $\mu$ on $M$ by $\mu(a\leq x)=\lambda_0(f^{-1} ([a,+\infty)))$ and $\mu(aEx)=0$ for all $a\in M$. This measure is generically stable. It lifts the generically stable generic type of the imaginary sort $M/E$. We will see that this phenomenon is general: we can always lift generically stable types in an imaginary sort to generically stable measures on the main sort (see the proof of Lemma \ref{im}).

Of course this example is rather special in that the generic type of $M/E$ is stable and not just generically stable. Here is an example where this is not the case. Start with the structure with universe $\mathbb Q$ and with language $\{P_n(x,y) : n<\omega\}$ where $P_n(x,y)$ holds if and only if $|x-y|<n$. Then there is a generically stable type $p$ at infinity. Now expand that structure with a generic linear order $<$ (such that every infinite definable set in the reduct is dense-co-dense with respect to $<$). Then the type $p$ lifts in many ways to a generically stable measure, by taking the union of it with some Lebesgue measure on $<$. We thus obtain a non-smooth generically stable measure, but there are no generically stable types in the expanded theory, even in $M^{eq}$.
\\

We will use standard notation. We will work with a complete first order theory $T$ in some language $L$; $T$ is assumed to be $NIP$ throughout the paper. For simplicity, we assume that $T$ is one-sorted and work in $T^{eq}$. We have a monster model $\bar M$; $M$, $N$,... will denote small submodels of $\bar M$, and $A,B,C...$ small parameter sets. We will not distinguish between points and tuples; they will be named by $a,b,c...$ and $x,y,z...$ will designate variables of finite or infinite tuples. The notation $L_x(A)$ denotes the set of formulas with parameters in $A$ and free variable $x$.

The space of types over $A$ in variable $x$ is designated by $S_x(A)$. It is equipped with the usual compact topology and the associated $\sigma$-algebra of Borel subsets. By ``$X$ is Borel over $A$", we mean that it is a Borel subset of some $S_x(A)$. We write $a \equiv_M b$ for $\tp(a/M)=\tp(b/M)$.

By a \emph{global} type or measure, we mean a type or a measure over $\bar M$.

\section{Preliminaries}

We recall some basic facts about Keisler measures. We will be brief, and the reader is refered to \cite{NIP2} and \cite{NIP3} for more details.

We make a blanket assumption that $T$ is $NIP$.

\subsubsection*{Basic definitions}
A \emph{Keisler measure} (or simply a \emph{measure}) over $A$ in variable $x$ is a finitely additive probability measure on the boolean algebra $L_x(A)$ of formulas with free variable $x$ and parameters in $A$. As in section 4 of \cite{NIP2}, such a measure extends uniquely to a regular Borel probability ($\sigma$-additive) measure on the type space $S_x(A)$. (``Regular" means that the measure of any Borel set $X$ is the infimum of the measures of open sets $O$ such that $X\subseteq O$. Furthermore, in our situation, working on a totally disconnected space, the measure of $O$ is itself the supremum of measures of clopen sets inside it.)

Conversely, given a regular Borel measure on $S_x(A)$, its restriction to the clopen sets gives a Keisler measure.

We will denote by $\mathcal M_x(A)$ the space of Keisler measures over $A$ and often write $\mu_x \in \mathcal M(A)$ for $\mu \in \mathcal M_x(A)$, keeping track of the variable in the name of the measure. We can consider $\mathcal M_x(A)$ as a subset of $[0,1]^{L_x(A)}$. It inherits the product topology making it a compact Hausdorff space.

\begin{lemme}\label{boolpartial}
Let $\Omega \subseteq L_x(A)$ be a set of formulas closed under intersection, union and complement and containing $\top$. Let $\mu_0$ be a finitely additive measure on $\Omega$ with values in $[0,1]$ such that $\mu_0(\top)=1$. Then $\mu$ extends to a Keisler measure over $A$.
\end{lemme}
\begin{proof}
By compactness in the space $[0,1]^{L_x(A)}$, it is enough to show that given $\psi_1(x),..,\psi_n(x)$ formulas in $L_x(A)$, there is a function $f=\langle \psi_1,..,\psi_n \rangle \rightarrow [0,1]$ finitely additive and compatible with $\mu_0$ (where $\langle B \rangle$ denotes the boolean algebra generated by $B$). We may assume that $\psi_1,..,\psi_n$ are the atoms of the boolean algebra $B$ that they generate.

The elements of $\Omega$ in $B$ form a sub-boolean algebra. Let $\phi_1,..,\phi_m$ be its atoms. We have say:
$$ \phi_1 = \psi_1 \vee ... \vee \psi_{i_1} \quad \phi_2 = \psi_{i_2} \vee ... \vee \psi_{i_3} \quad \text{etc.}$$
Then any finitely additive $f$ satisfying $f(\psi_1)+...+f(\psi_{i_1}) = \mu_0(\phi_1)$ etc. will do. 
\end{proof}

Here are some basic definitions: 

\begin{defi}
Let $M \prec N$, with $N$ $|M|^+$-saturated and let $\mu_x \in \mathcal M(N)$,
\begin{itemize}
\item $\mu$ is \emph{finitely satisfiable} in $M$ if for every $\phi \in L_x(N)$ such that $\mu(\phi)>0$, there is $a\in M$ such that $N\models \phi(a)$.
\item $\mu$ is $M$-\emph{invariant} if for every $\phi(x;y) \in L$, and $b\equiv_M b'$, $\mu(\phi(x;b))=\mu(\phi(x;b'))$.
\item $\mu$ is \emph{definable} over $M$ if it is $M$-invariant and for every $\phi(x;y) \in L$, and $r\in [0,1]$, the set $\{ p\in S_y(M): \mu(\phi(x;b))\leq r$ for any $b\in N, b\models p \}$ is a closed set of $S_y(M)$.
\item $\mu$ is Borel-definable over $M$ if the above set is a Borel set of $S_y(M)$.
\end{itemize}
\end{defi}

\begin{prop}[\cite{NIP2} 4.9]
If $\mu \in \mathcal M(N)$ is $M$-invariant ($N$ is $|M|^+$-saturated), then it is Borel definable over $M$.
\end{prop}

The $\emph{support}$ of $\mu\in \mathcal M(N)$ is the set of types $p\in S(N)$ weakly random for $\mu$, {\it i.e.}, such that $p\vdash \neg \phi(x)$ for every $\phi(x)\in L(N)$ such that $\mu(\phi(x))=0$. We will denote the support of $\mu$ by $S(\mu)$; it is a closed set of $S(N)$. Note that if $\mu$ is finitely satisfiable in $M$ then every type in $S(\mu)$ is also finitely satisfiable in $M$.

The next proposition says that in $NIP$ theories, measures can be well approximated by averages of types. We use the notation $Fr(X;a_1,..,a_n)$ which stands for $\frac 1 n |\{i: a_i\in X\}|$.
\begin{prop}[\cite{NIP3} 2.8]\label{VCappr}
Let $\mu_x$ be any measure over $M$.
 Let $\phi(x,y)\in L$, $\epsilon>0$,
 and let $X_1,..,X_k$ be Borel sets over $M$. Then for all large enough $n$ there are $a_1,..,a_n$ such that for all 
$r=1,..,k$ and all $b\in M$, $\mu(X_r \cap \phi(x,b))$ is within $\epsilon$
 of $F r(X_r \cap \phi(x, b); a_1,..,a_n )$.
\end{prop}

Taking $X_1=S(\mu)$ in the above proposition, we see that we can always assume that $\tp(a_i/M)\in S(\mu)$ for every index $i$.

\subsubsection*{Invariant extensions}
As in the case of types, the study of Keisler measures differs from measure theory in that the space in two dimensions is not the product of the one dimensional spaces, and there are in general different ways to amalgamate two measures in one variable to form a measure in two variables. We recall here the basic construction of invariant extensions.

Let $M\prec N$, $N$ being $|M|^+$-saturated, and $\mu_x \in \mathcal M(N)$ be $M$-invariant. If $\lambda_y\in \mathcal M(N)$ is any measure, then we can define the \emph{invariant extension} of $\mu_x$ over $\lambda_y$, denoted $\mu_x \otimes \lambda_y$. It is a measure in the two variables $x,y$ defined the following way. Let $\phi(x,y) \in L(N)$. Take a small model $P\prec N$ containing $M$ and the parameters of $\phi$. Define
$\mu_x \otimes \lambda_y (\phi(x,y)) = \int f(y) d\lambda_y,$ the integral ranging over $S_y(P)$
where $f(p) = \mu(\phi(x,b))$ for $b\in N$, $b\models p$ (this function is Borel-measurable by Borel-definability of $\mu$). It is easy to check that this does not depend on the choice of $P$.

If $\lambda_y$ is also invariant, we can also form the product $\lambda_y \otimes \mu_x$. In general it will not be the case that $\lambda_y \otimes \mu_x=\mu_x \otimes \lambda_y$.
\\

If $\mu_x$ is a global $M$-invariant measure, we define by induction: $\mu^{(n)}_{x_1...x_n}$ by $\mu^{(1)}_{x_1}=\mu$ and $\mu^{n+1}_{x_1...x_{n+1}} = \mu_{x_{n+1}} \otimes \mu^{(n)}_{x_1...x_n}$. We let $\mu^{(\omega)}_{x_1x_2...}$ be the union and call it the \emph{Morley sequence} of $\mu$. It is an indiscernible sequence in the following sense.

\begin{defi}
A measure $\mu_{x_1x_2...}$ is indiscernible over $A$ if for every $\phi(x_1,..,x_n) \in L(A)$ and indices $i_1<...<i_n$, we have $$\mu(\phi(x_1,..,x_n))=\mu(\phi(x_{i_1},..,x_{i_n})).$$
\end{defi}

We define in the same way $\mu_{x_1x_2...}$ to be \emph{totally indiscernible} by removing in the above definition the assumption that the indices $i_1,..,i_n$ are ordered.

We will need the following result from \cite{BY1} (see also \cite{NIP3}, 2.10).

\begin{prop}\label{by}
If $\mu_{x_1,x_2,...}\in \mathcal M(M)$ is indiscernible and $\omega_{y,x_1,x_2,...}$ extends $\mu$, then for every formula $\phi(x,y)\in L(M)$, $\lim_{i\rightarrow \omega} \omega(\phi(x_i,y))$ exists. Equivalently, for any $\phi(x,y)$ and $\epsilon>0$, there is $N$ such that for any measure $\omega_{y,x_1,x_2,...}$ as above, we have $|\omega(\phi(x_i,y)) -\omega(\phi(x_{i+1},y))| \geq \epsilon$ for at most $N$ values of $i$.
\end{prop}

\subsubsection*{Generically stable measures}

This paper is concerned with building generically stable measures. They are measures that behave very much like types in a stable theory (at least as far as non-forking extensions are concerned). Generically stable \emph{types} were defined in \cite{NIP2} and this definition was naturally extended to measures in \cite{NIP3}. We recall here some equivalent definitions (a few others are given in \cite{NIP3}, Theorem 3.3).

\begin{thm}[Generically stable measure]\label{genstable}
Let $\mu_x$ be a global $M$-invariant measure. Then the following are equivalent:
\begin{enumerate}
\item $\mu$ is both definable and finitely satisfiable (necessarily over $M$),
\item $\mu^{(\omega)}|_M$ is totally indiscernible,
\item for any global $M$-invariant Keisler measure $\lambda_y$, $\mu_x \otimes \lambda_y=\lambda_y \otimes \mu_x$,
\item $\mu$ commutes with itself: $\mu_x \otimes \mu_y=\mu_y\otimes \mu_x$. 
\end{enumerate}
If $\mu$ satisfies one of those properties, we say it is \emph{generically stable}.
\end{thm}

Let $\mu_x$ be a measure, and $f$ a definable map whose range is the sort of the variable $x$. Then one can consider the push-forward measure $\lambda_y=f_*(\mu)$ defined by $\lambda_y(\phi(y))=\mu_x(\phi(f(x)))$. This is again a Keisler measure. If $\mu$ is definable over $M$ (resp. finitely satisfiable in $M$) and $f$ is $M$-definable, then $f_*(\mu)$ is again definable over $M$ (resp. finitely satisfiable in $M$). In particular, if $\mu$ is a global generically stable measure, then $f_*(\mu)$ is also generically stable.

\subsubsection*{Smooth measures}

\begin{defi}[Smooth]
A measure $\mu \in \mathcal M(N)$ is smooth if $\mu$ has a unique global extension. If $M\subset N$, we will say that $\mu$ is smooth over $M$ is $\mu|_M$ is smooth.
\end{defi}

The following important properties are proved in \cite{NIP3}.

\begin{prop}[\cite{NIP3}, 2.3]\label{smoothlemma}
Let $\mu_x$ be smooth over $M$ and let $\phi(x,y) \in L$ and $\epsilon >0$. Then there are formulas $\nu_i^1(x)$, $\nu_i^2(x)$ and $\psi_i(y)$ for $i=1,..,n$ in $L(M)$ such that: 
\begin{enumerate}
\item The formulas $\psi_i(y)$ form a partition of the $y$-space,
\item for all $i$ and $b\in \bar M$, if $\psi_i(b)$ holds, then $\bar M \models \nu_i^1(x) \rightarrow \phi(x,b) \rightarrow \nu_i^2(x)$,
\item for each $i$, $\mu(\nu_i^2(x)) - \mu(\nu_i^1(x)) < \epsilon$.
\end{enumerate}
\end{prop}

Note that conversely, if the conclusion holds for all $\phi(x,y)$ and $\epsilon$, then $\mu$ is smooth.

\begin{cor}\label{smoothcor}

If $\mu$ is smooth over $N$, then:
\begin{enumerate}
\item
there is $M\prec N$ of size $|T|$ such that $\mu$ is smooth over $M$,
\item
$\mu$ is definable and finitely satisfiable in $N$ (in particular $\mu$ is generically stable),
\item
if $\lambda_y$ is a measure over $N$, then there is a unique separated amalgam $\omega_{xy}$ of $\mu_x$ and $\lambda_y$ (see Definition \ref{separ}).
\end{enumerate}
\end{cor}

Will also need the following fact (initially from \cite{Keisler}), that we consider as a way to \emph{realize} measures.

\begin{lemme}[\cite{NIP3}, 2.2]\label{extsmooth}
Let $\mu_x$ be a measure over $M$. Then there is an extension $M\prec N$ and a measure $\mu'_x$ over $N$ extending $\mu$ such that $\mu'$ is smooth.
\end{lemme}

Some additional facts about smooth measures can be found in Section \ref{smoothsec}.

\section{Amalgams}

We fix throughout this section the following objects: $M\prec N$ two models, with $N$ being $|M|^+$-saturated, and $\mu_x,\lambda_y$ two measures over $N$. An \emph{amalgam} of $\mu_x$ and $\lambda_y$ is simply a measure $\omega_{xy}$ extending $\mu_x \cup \lambda_y$. We are interested in the characterization of different possible amalgams, especially `free' amalgams.

The most basic property of an amalgam is independence in the sense of probability theory, which we call separation.
\begin{defi}[Separated]\label{separ}
The amalgam $\omega_{xy}$ is separated if $$\omega(\phi(x) \wedge \psi(y))=\mu(\phi(x)).\lambda(\psi(y))$$ for all $\phi(x), \psi(y) \in L(N)$.
\end{defi}

Note that if $\mu$ is $M$-invariant, then $\mu \otimes \lambda$ is separated. Also if $\mu$ or $\lambda$ is a type, then every amalgam is separated.
\\

We now go on to define when the amalgam $\omega_{xy}$ is a finitely satisfiable extension of $\mu$. A natural attempt would be to ask for example that $\omega(\theta(x,y)) \leq \sup_{a\in M} \lambda(\theta(a,y))$. However, this seems to be too weak, and we will ask for something stronger, allowing to `localize' on some clopen $\phi(x)$.

\begin{defi}[f.s.\,extension]
The amalgam $\omega_{xy}$ is an f.s.\,extension in $M$ of $\mu$ over $\lambda$ if the following holds for all $\theta(x,y), \phi(x) \in L(N)$:
$$\omega(\theta(x,y)\wedge \phi(x)) \leq \mu(\phi(x)).\sup_{a\in \phi(M)} \lambda(\theta(a,y)).$$
\end{defi}

First some basic observations. If $\omega$ is an f.s.\,extension in $M$ of $\mu$, then it is a separated amalgam (apply the definition with $\theta(x,y) = \psi(y)$ and then $\theta(x,y)=\neg \psi(y)$). Also, the existence of such an amalgam implies that $\mu$ itself is finitely satisfiable in $M$ (hence $M$-invariant).

Assume $\lambda$ is a type realized by some $a$. We can view $\omega_{xy}$ as $\omega'_x \in \mathcal M(Na)$. Then $\omega_{xy}$ is an f.s. extension in $M$ of $\mu$ if and only if $\omega'_x$ is finitely satisfiable in $M$.

In the following propositions, we use the notation $Av(X_i: i=1,..,n)$ to mean $\frac 1 n|\{i: X_i \text{ holds }\}|$.
\begin{prop}
Assume that $\mu$ is finitely satisfiable in $M$, then $\omega_{xy}=\mu_x \otimes \lambda_y$ is an f.s.\,extension (in $M$) of $\mu$ over $\lambda$.
\end{prop}
\begin{proof}
We first assume that $\mu_x = p_x$ is a type. Let $\theta(x,y) \in L(N)$ and $\phi(x) \in L(N)$ such that $p\vdash \phi(x)$, and take a small model $P \subset N$ containing $M$ and the parameters of $\theta$.

Then $\omega(\theta(x,y)) = \lambda(B)$ where $B\subseteq S(N)$ is the ($P$-invariant) Borel subset $B=\{q: p \vdash \theta(x,b) $ for some (any) $b\models q \}$. Let $\epsilon >0$. By $NIP$ (see \ref{VCappr}), there are points $b_1,..,b_n \in N$ such that:
$$|\lambda(B) - Av(b_i\in B: i=1..n)| \leq \epsilon,$$
$$\forall a\in P, ~|\lambda(\theta(a,y))-Av(\theta(a,b_i): i=1...n)| \leq \epsilon.$$

As $p$ is finitely satisfiable in $M$, there is $a_0 \in \phi(M)$ such that for every $i \in \{1,..,n\}$: 
$$p\vdash \theta(x,b_i) \leftrightarrow N \models \theta(a_0,b_i).$$

Now $p \vdash \theta(x,b_i) \iff b_i \in B$ so: 
\begin{eqnarray*}
\lambda(B) & \approx & Av(b_i \in B) \\
& = & Av(p \vdash \theta(x,b_i)) \\
& = & Av( \theta(a_0,b_i)) \\
& \approx & \lambda(\theta(a_0,y)).
\end{eqnarray*}
\noindent(Where $x \approx y$ means $|x-y|\leq \epsilon$.)

So $|\lambda(B) - \lambda(\theta(a_0,y))| \leq 2\epsilon$. As this is true for all $\epsilon >0$, and remembering $\lambda(B) = \omega(\theta(x,y))$, we deduce $\omega(\theta(x,y)) \leq \sup_{a \in \phi(M)} \lambda(\phi(a,y))$. This finishes the proof in the cas $\mu=p$.
\\

We now consider the general case.
Let as above $\theta(x,y), \phi(x) \in L(P)$.
Let $\epsilon>0$. By Proposition \ref{VCappr} we can find $p_1,…,p_n \in S(\mu)\subset S(N)$ such that:
$$\forall b \in N, |\mu(\theta(x,b))-Av(p_i\vdash \theta(x,b))|\leq\epsilon$$
and: 
$$|\mu(\phi(x))-Av(p_i \vdash\phi(x))| \leq \epsilon.$$
Note that $p_1,…,p_n$ are finitely satisfiable in $M$.
Let, for $b\in M$, $f(b) = \mu(\theta(x,b) \wedge \phi(x))$ and $f_n(b) = \frac{1}{n}. Card\{k:p_k\vdash \theta(x,b)\wedge\phi(x)\}$. Let $m$ be the number of indices $i$ for which $p_i \vdash \phi(x)$.

Then, \begin{eqnarray*}\nonumber \omega(\theta(x,y) \wedge \phi(x)) &=&
 \int f(y)d\lambda_y \\ \nonumber &\leq& \int f_n(y)d\lambda_y + \epsilon \\\nonumber
&\leq& \frac m n \sup_k \lambda(\{b\mid p_k \vdash \theta(x,b)\wedge \phi(x)\})+\epsilon \\\nonumber
&\leq& \frac m n\sup_{a\in \phi(M)} \lambda(\theta(a,y))+\epsilon
\\\nonumber &\leq& \mu(\phi(x))\sup_{a\in \phi(M)} \lambda(\theta(a,y))+2\epsilon. \end{eqnarray*}
(We use the first part of the proof to go from line 3 to 4).

As $\epsilon$ was arbitrary, we are done.
\end{proof}

We now show that if $\mu$ is finitely satisfiable in $M$, then the invariant extension is the only f.s.\,extension of $\mu$ over $\lambda$.

\begin{prop}\label{fsunique}
Assume that $\mu$ is finitely satisfiable in $M$ and $\omega_{xy}$ is an f.s.\,extension in $M$ of $\mu_x$, then $\omega_{xy}=\mu_x\otimes \lambda_y$.
\end{prop}
\begin{proof}
Let $\theta(x,y) \in L(N)$ and let $\epsilon >0$. Let $P\subset N$ be a small model containing $M$ and the parameters of $\theta$. By \ref{VCappr} we can find $b_1,…,b_n \in N$ such that $$|\lambda(\theta(a,y))-Av(\theta(a,b_i))| \leq \epsilon \text{, for all }a \in P.$$

For all $k \in \{0,…,n\}$, let $B_k(x)$ be the formula saying: ``there are exactly $k$ values of $i$ for which $\theta(x,b_i)$ is true".

Then: \begin{align*} \omega(\theta(x,y))&= \sum_k \omega(\theta(x,y) \wedge B_k(x))
\\ &\leq \sum_k \mu(B_k(x)).\sup_{a \in B_k(M)} \lambda(\theta(a,y))
\\ & \leq \sum_k \frac k n \mu(B_k(x)) + \epsilon. \end{align*}
(We use finite satisfiability of $\mu$ on line 2). 

Similarly, $$\omega(\neg \theta(x,y)) \leq \displaystyle \sum_k \left( 1 - \frac k n \right ) \mu(B_k(x)) + \epsilon,$$ and therefore $$\left |\omega(\theta(x,y)) - \sum_k \frac k n \mu(B_k(x))\right | \leq \epsilon.$$

Letting $\epsilon \rightarrow 0$, we see that $\omega(\theta(x,y))$ is uniquely determined by $\mu$ and $\lambda$ and the fact that $\omega$ is an f.s.\,extension of $\mu$. By the previous proposition, we have $\omega_{xy}(\theta(x,y)) = \mu_x \otimes \lambda_y(\theta(x,y))$.
\end{proof}
\section{Symmetrizations}

The following construction already appears in \cite{NIP3}. Let $I=\langle a_t: t\in[0,1] \rangle$ be an indiscernible sequence indexed by $[0,1]$ (we will call this an \emph{indiscernible segment}). If $\phi(x,y)$ is a formula and $b\in \bar M$, then $NIP$ implies that $\lfloor\phi(x,b)\rfloor:=\{t \in [0,1]: \models \phi(a_t,b)\}$ is a finite union of intervals and points. Let $m$ denote the Lebesgue measure on $[0,1]$. Then we can define a global measure $\mu=Av(I)$ by $\mu(\phi(x,b))=m(\lfloor\phi(x,b)\rfloor)$. It is called the average measure of $I$.

\begin{lemme}
For any indiscernible segment $I\subset M$, the average measure $\mu=Av(I)$ is generically stable over $M$. Furthermore, for any $n$ and $\phi(x_1,..,x_n)\in L(\bar M)$, $\mu^{(n)}(\theta(x_1,..,x_n))=\int_{t_1\in [0,1]}..\int_{t_n\in [0,1]} \theta(x_{a_1},..,x_{a_n}) d t_1...d t_n$.
\end{lemme}
\begin{proof}
First notice that $\mu$ is finitely satisfiable in $M$ by construction. It is easy to check directly from the definition that the formula given for $\mu^{(n)}$ is valid. From this, it follows that $\mu^{(n)}$ is symmetric for all $n$, therefore $\mu$ is generically stable.
\end{proof}

If $p \in S(M)$ is any type. Let $\tilde p$ be an $M$-invariant global extension of $p$ (for example, a coheir). Let $I$ be a Morley segment of $\tilde p$ (i.e., a Morley sequence indexed by $[0,1]$). Then $Av(I)$ is a generically stable measure and extends $p$. Note that if $p$ is already generically stable, then $\mu=p$. In general, by the previous lemma, the Morley sequence of $\mu$ is a \emph{symmetrization} of the Morley sequence of $p$, i.e. $\mu^{(n)}|_M$ is the average over all permutation of variables of $p^{(n)}|_M$:
$$\mu^{(n)}(\theta(x_1,..,x_n))=\frac 1{n!} \sum_{\sigma \in \mathfrak S_n} p^{(n)}( \theta(x_{\sigma(1)},..,x_{\sigma(n)})),$$
for $\theta(x_1,..,x_n)\in L(M)$. For this reason, we will call $\mu$ a symmetrization of $p$.
\\

Our aim now is to define the same construction starting with a measure instead of a type $p$. So let $\mu\in \mathcal M(M)$ be any measure. Take $\tilde \mu$ a global $M$-invariant extension of $\mu$. We have defined the Morley sequence $\tilde \mu^{(\omega)}_{x_1x_2...}$ of $\tilde \mu$ and we can analogously define its Morley segment $\tilde \mu^{([0,1])}_{\bar x}$ where $\bar x$ stands for $\langle x_t:t\in [0,1]\rangle$. Now let $\nu_{\bar x}$, be a smooth extension of $\tilde \mu^{([0,1])}$ (this is the analogue of taking a realization of a Morley segment of $p$). For any $b\in \bar M$, consider the function $f_{\phi(x,b)}: t \mapsto \nu(\phi(x_t,b))$. By Proposition \ref{by}, this function has only countably many points of discontinuity. It is therefore integrable on $[0,1]$.

We can consider the integral $\mu_{\Sigma}(\phi(x,b))=\int_{t\in [0,1]} \nu(\phi(x_t,b)) dt$. This defines a global Keisler measure extending $\mu$.

\begin{prop}
The measure $\mu_{\Sigma}$ constructed above is generically stable. Furthermore, for every $n$, and $\theta(x_1,..,x_n)\in L(M)$, $$\mu_{\Sigma}^{(n)}(\theta(x_1,..,x_n))= \frac 1 {n!} \sum_{\sigma \in \mathfrak S_n}\tilde \mu^{(\omega)}(\theta(x_{\sigma(1)},..,x_{\sigma(n)})).$$
\end{prop}
\begin{proof}
Let $P$ be a small model containing $M$ over which $\nu$ is smooth. Then $\nu$ is finitely satisfiable in $P$, and it follows that $\mu_{\Sigma}$ is also finitely satisfiable in $P$. 
Define an $n$-ary global measure $\lambda^n_{x_1..x_n}$ by 
$$\lambda^n(\theta(x_1,..,x_n))= \int_{t_1\in [0,1]}..\int_{t_n \in [0,1]} \nu(\theta(x_{t_1},..,x_{t_n}) )dt_1..dt_n,$$
for any formula $\theta(x_1,..,x_n)\in L(\bar M)$.

We will show by induction that $\lambda^n=\mu_{\Sigma}^{(n)}$ for all $n$. The second assertion of the proposition will follow immediately by direct computation (remembering $\tilde \mu^{(\omega)}|_M = \nu |_M$). Then the first assertion follows since the expression given for $\mu^{(n)}_{\Sigma}$ is symmetric. We show that $\lambda^n$ defines a f.s.\,extension in $P$ of $\mu_{\Sigma}$, and is therefore, by \ref{fsunique} and induction equal to $\mu^{(n)}_{\Sigma}$.

For simplicity of notations, we write the details only for $n=2$ (the case $n=1$ being true by definition). Let $\theta(x,y),\phi(x) \in L(\bar M)$. The transformations are explained below.

\begin{align*} \lambda^2(\theta(x,y)\wedge \phi(x))& = \int_{t\in [0,1]}\int_{t'\in[0,1]} \nu_{\bar x}(\theta(x_{t},x_{t'})\wedge \phi(x_t))dt' dt
\\ &=\int_{t\in [0,1]} \mu_{\Sigma,y} \otimes {\nu_{\bar x}}(\theta(x_t,y) \wedge \phi(x_t)) dt
\\ &\leq \sup_{a \in \phi(P)} \mu_{\Sigma,y}(\theta(a,y)) \times \int_{t \in [0,1]}\nu_{\bar x}(\phi(x_t))dt
\\ & \leq \mu_{\Sigma,x}(\phi(x)). \sup_{a \in \phi(P)} \mu_{\Sigma,y}(\theta(a,y)).
\end{align*}

Explanation: $\nu$ is a smooth measure, so by \ref{smoothcor}, it admits a unique separated amalgam with any other measure. In particular with $\mu_{\Sigma}$. The measure $\omega_{\bar xy}$ defined by $\omega(\theta(\bar x,y))=\int_{t'\in[0,1]} \nu(\theta(\bar x,x_{t'}))dt$ is such an amalgam. Therefore it is equal to $\mu_{\Sigma,y}\otimes \nu_{\bar x}$. This justifies the second line.

As $\mu_{\Sigma}$ is finitely satisfiable in $P$, $\mu_{\Sigma,y}\otimes \nu_{\bar x}$ is an f.s.\,extension of $\mu_{\Sigma}$ over $\nu$ (in $P$); this explains the third line. The forth line is just the definition of $\mu_{\Sigma}$.

\end{proof}

If $\mu$ is a measure over $M$, a \emph{symmetrization} of $\mu$ is a global extension $\mu_{\Sigma}$ of $\mu$ obtaining by applying the construction above with some $\tilde \mu$ and $\nu$. If $\mu$ is a global $M$-invariant measure, we will say that $\mu_{\Sigma}$ is a symmetrization of $\mu$ over $M$ if it is obtained by the construction above where $\mu$ and $\tilde \mu$ there stand for $\mu|_M$ and $\mu$ respectively.

Note that if we build $\mu_{\Sigma}$ from a Morley sequence of $\tilde \mu$ over $M$, then $\mu_{\Sigma}$ is in general not $M$-invariant. In fact, it is $M$-invariant if and only if $\mu$ is generically stable, in which case $\mu_{\Sigma}=\mu$ (as explained below).

\begin{prop}

Let $\mu_x$ be a global $M$-invariant measure, and $\mu_{\Sigma}$ a symmetrization of $\mu$ over $M$. Let also $f$ be an $M$-definable function whose domain and range is in the same sort as the variable $x$.
\begin{enumerate}
\item If $\mu$ is generically stable, then $\mu_{\Sigma}=\mu$,
\item $f_*(\mu_{\Sigma})$ is a symmetrization of $f_*(\mu)$,
\item If $f_*(\mu)=\mu$, then $f_*(\mu_{\Sigma})=\mu_{\Sigma}$.
\end{enumerate}
\end{prop}
\begin{proof}
1. If $\mu$ is generically stable, then $\mu^{(\omega)}$ is totally indiscernible. It follows by \ref{by} that for every $\phi(x,c)\in L(\bar M)$ and $\epsilon>0$, the set $\{t\in [0,1]: |\nu(\phi(a_t,c))-\tilde \mu(\phi(a,c))|>\epsilon\}$ is finite. Therefore the definition of $\mu_\Sigma$ implies that $\mu_\Sigma =\mu$.

2. Clear: $f_*(\mu_\Sigma)$ is a symmetrization of $f_*(\mu)$ built using $f_*(\tilde \mu)$ and $f_*(\nu)$.

3. Let $\nu_{\bar x}$ be as in the construction of $\mu_{\Sigma}$. For $I\subseteq [0,1]$, define the measure $\nu^I_{\bar x}$ by $\nu^I(\phi(x_{t_1},..,x_{t_n}))=\nu^I(\phi(f^I_{t_1}(x_{t_1}),..,f^I_{t_n}(x_{t_n})))$ where $\phi(x_1,..,x_n)\in L_{\bar x}(\bar M)$ and $f^I_{t}=f$ if $t\in I$ and is the identity otherwise.

\vspace{4pt}
\underline{Claim}: For every $I$, $\nu^I$ is the measure of an $M$-indiscernible segment.
\vspace{4pt}
\begin{proof}
The claim concerns only the restriction to $M$ (indeed to $\emptyset$) of the measure $\nu^I$. Note that $\nu|_M$ is just $\mu^{([0,1])}|_M$, and as $f$ is $M$ definable, the property we need to check does not depend on the choice of the smooth extension $\nu$. It is enough to show that $\mu^{(\omega)} (\theta(x_1,..,x_n))=\mu^{(\omega)}(\theta(g_1(x_1),..,g_n(x_n)))$, for every $\theta(x_1,...,x_n)\in L(M)$, where $g_i$ is either $f$ or the identity. This is easily checked by direct computation and induction on $n$. For example, to see that $\mu^{(2)}(\theta(f(x_1),f(x_2)))=\mu^{(2)}(\theta(x_1,x_2))$, write
\begin{eqnarray*} \mu^{(2)}_{x_1x_2}(\theta(f(x_1),f(x_2))&=& \int \mu_{x_2}(\theta(f(a),f(x_2)))) d\mu_a\\
&=& \int f_*\mu_{x_2}(\theta(b,x_2)) df_*\mu_b\\
&=& \int \mu_{x_2}(\theta(b,x_2)) d\mu_b\\
&=& \mu_{x_1x_2}^{(2)}(\theta(x_1,x_2)).\end{eqnarray*}
\newcommand{\halmos}{\rule{1.5mm}{2.5mm}} 
\renewcommand{\qedsymbol}{\halmos} 

\end{proof}

It follows by Proposition \ref{by} that for every $\phi(x)\in L(\bar M)$ and $\epsilon$, there can be only finitely many values of $t\in [0,1]$ for which $|\nu(\phi(x_t))-\nu(\phi(f(x_t)))|>\epsilon$. Hence the average of $\nu_{\bar x}$ is the same as the one of $\nu^I_{\bar x}$ for each $I$. As $f_*(\mu_{\Sigma})$ is the average of $\nu^I$ for $I=[0,1]$, we have $f_*(\mu_{\Sigma})=\mu_{\Sigma}$.
\end{proof}

As an application, we give a short proof of a result from \cite{NIP2}, section 7. (It is not stated explicitly there, as the notion of generically stable measure had not been introduced yet, but is the content of the pages from Lemma 7.1 to Lemma 7.6.) If $G$ is a definable group and $M$ a model, by a measure $\mu_x$ being $G(M)$-invariant, we mean that $\mu$ concentrates on $G$ (i.e. $\mu(x\in G)=1$) and for each $g\in G(M)$ and $\phi(x)$ a formula, $\mu(\phi(x))=\mu(\phi(g.x))$.

\begin{prop}
Let $M$ be a model,  and $G$ a definable group. Assume there is $\mu\in \mathcal M(M)$ a $G(M)$-invariant measure. Then $\mu$ extends to a global generically stable $G(M)$-invariant measure.
\end{prop}
\begin{proof}
First, find a global extension $\tilde \mu$ of $\mu$ that is $G(M)$-invariant, and $M_1$-invariant for some small $M_1\supset M$. This can be done with Keisler's smooth measure construction, see \cite{NIP2}, 7.6. Let $\mu_{\Sigma}$ be a symmetrization of it over $M_1$. Then by the previous propositions, $\mu_{\Sigma}$ is generically stable and $G(M)$-invariant.
\end{proof}

\section{Smooth measures}\label{smoothsec}

The goal of this section is to prove that generically stable measures in o-minimal theories or in the theory of the p-adics are smooth.

We start by giving a characterization of smoothness which will be useful for proving that measures are smooth. Let $M\models T$ and let $\phi(x,y)\in L(M)$. For $b \in \bar M$, we define the \emph{border} of $\phi(x,b)$ (over $M$) as $\partial_x^M\phi(x,b) = \{p\in S_x(M): $ there are $a,a' \models p$ such that $\phi(a,b) \wedge \neg \phi(a',b)$ holds $\}$. This is a closed subset of the space of types $S_x(M)$. We will often omit $x$ and $M$ in the notation. Note that $\partial^M\phi(x,b)$ depends only on $q=\tp(b/M)$, so we will also sometimes write $\partial^M\phi(x,q)$ for $\partial^M\phi(x,b)$.

We will also be needing the notion of a \emph{localization} of a measure $\mu_x \in \mathcal M(M)$. Let $F\subseteq S_x(M)$ be a closed subspace such that $\mu(F)>0$. Then we define the localization of $\mu$ at $F$ by $\mu_F(\phi(x))=\mu(\phi(x) \cap F))/\mu(F)$ for all $\phi(x) \in L(M)$. This is again a Keisler measure over $M$.

\begin{lemme}
Let $\mu \in \mathcal M_x(N)$ and $F\subset S_x(N)$ a closed subset such that $\mu(F)>0$.
If $\mu$ is smooth, then $\mu_F$ is smooth.

\end{lemme}
\begin{proof}
Assume that $\mu$ is smooth.
Let $\nu$ be a global extension of $\mu_F$  and let $\tilde \mu$ be the unique global extension of $\mu$. Then we define a global measure $\nu'$ by $\nu'(\phi(x))=\tilde \mu(\phi(x) \cap F^{c})+\nu(\phi(x)\cap F).\mu(F)$. This defines an global extension of $\mu$. By smoothness, $\nu'=\tilde \mu$, and it follows that $\nu=\tilde \mu_F$.
\end{proof}

\begin{lemme}[Characterization of smoothness]\label{smoothcar}
Let $\mu\in \mathcal M_x(M)$. Then $\mu$ is smooth if and only if $\mu(\partial^M\phi(x,b))=0$ for all $\phi(x,y)\in L(M)$ and all $b\in \bar M$.
\end{lemme}
\begin{proof}
Let $\phi(x,y)$ and $b$ be as in the statement of the lemma. Let $O \subseteq S_x(M)$ be the set of types $p$ such that $p\vdash \phi(x,b)$. And let $F=\partial^M\phi(x,b)$. Then for any extension $\nu$ of $\mu$, we have $\mu(O)\leq \nu(\phi(x,b)) \leq \mu(O)+\mu(F)$. Therefore if $\mu(F)=0$ for all such $\phi$ and $b$, then $\mu$ is smooth.

Conversely, assume that $\mu$ is smooth and let $\phi(x,y)$ and $b$ be as above. Let $\epsilon >0$ and take $\nu_i^1(x)$, $\nu_i^2(x)$, $\psi_i(y)$, $i=1,..,n$ be as in Proposition \ref{smoothlemma}. Let $i$ be such that $\psi_i(b)$ holds. Then $\partial\phi(x,b) \subseteq \nu_i^2(x) \setminus \nu_i^1(x)$. Therefore $\mu(\partial\phi(x,b))<\epsilon$. As this is true for all $\epsilon>0$, $\mu(\partial\phi(x,b))=0$.
\end{proof}

To illustrate this, assume $T$ is o-minimal, let $M \prec N$ be models of $T$ and let $\phi(x)\in L(N)$ be a formula, $x$ a single variable. By o-minimality, $\phi(x)$ is a finite union of (closed or open) intervals. Let $a_0,..,a_{n-1}$ denote those end points that lie in $N \setminus M$. Then $\partial^M\phi=\{\tp(a_k/M): k<n\}$. In particular, it is finite.

\begin{lemme}\label{pf}
Let $\mu_x$ be a global measure, smooth over $M$. Let $f$ be an $M$ definable function whose range is the sort of the variable $x$. Then $f_*(\mu)$ is smooth.
\end{lemme}
\begin{proof}
Let $\mu_x$ be a global measure, smooth over $M$. Let $\lambda_y=f_*(\mu)$ and let $\phi(y)$ be an $M$-definable set. Let $F=\partial^{M} \phi(y)$. Define also $\psi(x)=\phi(f(x))$ and $G=\partial^M \psi(x)$. Then $\lambda(F)=\mu(G)=0$ as $\mu$ is smooth. By Lemma \ref{smoothcar}, $f_*(\mu)$ is smooth.
\end{proof}

\noindent
The following easy fact will be used implicitly in what follows.

\begin{lemme}\label{invsmooth}
Let $\mu$ be a global smooth measure. Assume that $\mu$ is $M$-invariant, then $\mu$ is smooth over $M$.
\end{lemme}
\begin{proof}
Proposition \ref{smoothlemma} gives us formulas $\psi_i(y,d)$, $\nu_i^1(x,d)$ and $\nu_i^2(x,d)$ with $d\in \bar M$ satisfying three properties as stated. It is enough to show that we can find $d'\in M$ such that the formulas $\psi_i(y,d')$, $\nu_i^1(x,d')$, $\nu_i^2(x,d')$ satisfy the same properties. Now the condition imposed on $d'$ by the first 2 properties is clopen. As $\mu$ is definable over $M$, the condition imposed by the third point is open. So we are looking for $d'$ in some open set of $S(M)$. As we know that this set is non-empty, it must intersect the set of realized types, and we find the required $d'$.
\end{proof}

\begin{lemme}\label{im}
Let $S_1,S_2$ be two sorts and $f:S_1 \rightarrow S_2$ a surjective definable map. Assume that all generically stable measures on $S_1$ are smooth, then it is also the case for $S_2$.
\end{lemme}
\begin{proof}
Let $\mu_x$ a generically stable measure concentrating on $S_2$ (i.e., $x$ is a variable of sort $S_2$). Let $M$ be a small model such that $\mu$ is $M$-invariant. We can find a measure $\eta_y$ on $\bar M$ which is $M$-invariant and such that $f_*(\eta)=\mu$ (this can be proven using \ref{boolpartial}. Take $\Omega_1$ to be definable sets of the form $f^{-1}(\psi)$, $\psi$ a definable set of $S_2$. Let $\Omega_2$ be composed of formulas $\phi(x)\in L(\bar M)$ that fork over $M$. Define $\mu_0$ on $\Omega=\llg \Omega_1,\Omega_2 \rrg$ by $\mu_0=f^{-1}(\mu)$ on $\Omega_1$ and $\mu_0=0$ on $\Omega_2$. By the Lemma extend $\mu_0$ to a global measure). Consider a symmetrization $\eta_{\Sigma}$ of $\eta$ over $M$. Then $f_*(\eta_{\Sigma})$ is a symmetrization of $\mu$ over $M$, and therefore is equal to $\mu$.

By hypothesis, $\eta_{\Sigma}$ is smooth. So $\mu=f_*(\eta_{\Sigma})$ is also smooth.
\end{proof}

We will use this ad-hoc criterion for smoothness.

\begin{prop}\label{transitive}
Assume $T$ has definable Skolem functions. Let $\mathcal S$ be a set of imaginary sorts containing the main sort $\bar M$. Assume that for any model $N$, any formula $\phi(\bar x,y)\in L(N)$ ($y$ a single variable from the main sort) is a boolean combination of formulas of the form $R(f(\bar x),y)$ where $R$ is a $\emptyset$-definable relation and $f$ is an $N$-definable function taking values in a sort from $\mathcal S$. Assume that for each $S\in \mathcal S$, all generically stable measures over $S$ are smooth.

Then any generically stable measure is smooth.
\end{prop}
\begin{proof}
Let $M\prec N$ be two models of $T$. Assume that for all $n$, all $n$ types over $M$ are realized in $N$. Let $\mu \in \mathcal M(N)$ be an $M$-invariant generically stable measure in $k$ variables. It is enough to show that any formula of the form $\phi(\bar x,c)$, $c\in \bar M$ a 1-tuple and $\phi(\bar x,y)\in L(N)$, has the same measure in any extension of $\mu$. (Because then $N(c)$ is a model of $T$ over which $\mu$ has a unique extension, and we can replace $N$ by it and iterate.) By hypothesis, we may assume that $\phi(\bar x,y)=R(f(\bar x),y)$ for some $R$ and $f$ as above.

If $\nu$ is a global extension of $\mu$, then $\nu(\phi(\bar x,c))=f_*(\nu)(R(z,c))$. Now $f_*(\nu)$ is generically stable and $N$-invariant. By hypothesis, it is smooth over $N$. Therefore $\nu(\phi(\bar x,c))$ is determined.
\end{proof}

\begin{cor}\label{omin}
If $T$ is o-minimal, then any generically stable measure is smooth.
\end{cor}
\begin{proof}
We will check the hypothesis of the previous proposition for $\mathcal S=\{\bar M\}$ (the main sort). Let $T$ be o-minimal. Every formula $\phi(\bar x,y)$ is a boolean combination of formulas of the form $y < f(\bar x)$ and $y=f(\bar x)$ where $f$ is a definable function. So the first part of the hypothesis is satisfied. Next, consider $\mu_x$ a global generically stable measure in dimension 1. Let $\phi(x)$ be a formula, $x$ a single variable, with parameters in some extension $\bar N \succ \bar M$. As explained after the proof of \ref{smoothcar}, $\partial^{\bar M}\phi$ is a finite set of non-realized types. Finiteness easily implies that all the types in $\partial^{\bar M}\phi$ are generically stable. As there are no non-realized generically stable types in $T$, $\mu(\partial^{\bar M}\phi)=0$. By \ref{smoothcar}, $\mu$ is smooth. Proposition \ref{transitive} therefore applies.
\end{proof}

\begin{cor}\label{padic}
Let $T=Th(\mathbb Q_p)$, for some $p$, then any generically stable measure is smooth.
\end{cor}
\begin{proof}
This case is similar to the o-minimal one. Let $\Gamma$ denote the value group. For $1 \geq k,n <\omega$, let $\mathfrak B_{k,n}$ be the set of canonical parameters of definable sets of the form $\{x: \val(x-a) \equiv k~[n]\}$, $a\in \bar M$ and $\mathfrak B$ be the set of canonical parameters of balls $\{x: \val(x-a) = \alpha\}$ for $a\in \bar M$, $\alpha \in \Gamma$.

We will check that Proposition \ref{transitive} applies with $\mathcal S=\bigcup_{k,n} \mathfrak B_{k,n}\cup \{\bar M, \mathfrak B\}$.

We leave it to the reader to check that all generically stable measures in one variable from $\bar M$ or from $\Gamma$ are smooth (this can be done as in the o-minimal case: check that the border of a definable set is finite).

Next, let $\mu_x$ be a generically stable measure on $\mathfrak B$. Let $\val$ denote the natural map from $\mathfrak B$ to $\Gamma$. Then, $\val_*(\mu)$ is generically stable and therefore is smooth. We may assume that $\val_*(\mu)$ is either a realized type or an atomless measure. Assume $\val_*(\mu) = ``x=\alpha"$ for some $\alpha \in \Gamma$. Then $\mu$ is a measure concentrating on $\mathfrak B_\alpha$: the sort of balls of radius $\alpha$. There is a surjective map $\pi: \bar M \rightarrow \mathfrak B_\alpha$, so by Lemma \ref{im}, $\mu$ is smooth. Now if $\val_*(\mu)$ is atomless, one can check by inspection that $\mu(\partial \phi(x))=0$ for every definable set $\phi(x)$.

Finally, for any $k,n <\omega$, there is a surjective map from $\mathfrak B$ to $\mathfrak B_{k,n}$, so all generically stable measures there are also smooth.

Let $\phi(x)\in L(A)$ be a definable set in dimension 1. Then $\phi(x)$ can be written as a finite boolean combination of formulas of the form $x \in b$ with $b$ in some $\mathfrak B_{k,n}$, $k,n <\omega$ or in $\mathfrak B$. We can choose the decomposition such that $b$ is $A$-definable. Therefore Proposition \ref{transitive} applies.
\end{proof}

Theories in which all generically stable measures are smooth will be studied in a subsequent work \cite{trans}, where equivalent characterizations will be given along with some properties. In particular, it will be shown that in a dp-minimal theory with no generically stable type in the main sort, all generically stable measures are smooth, generalizing the two corollaries above.

\section{$\sigma$-additive measures}

We prove in this section the main result of this paper: $\sigma$-additive measures are generically stable.
\\

Recall that if $M\models T$, an externally definable subset of $M$ is a subset of the form $\phi(M)$ where $\phi(x) \in L(\bar M)$. Assume that the model $M$ is equipped with a $\sigma$-algebra $\mathcal A$ such that externally definable sets are measurable. Then if $\mu_0$ is a ($\sigma$-additive) probability measure on $(M,\mathcal A)$, $\mu_0$ induces a global Keisler measure $\mu$. Namely $\mu(\phi(x))=\mu_0(\phi(M))$, for $\phi(x) \in L(\bar M)$.

\begin{thm}\label{sigma}
Let $T$ be $NIP$, $M\models T$ equipped with a $\sigma$-algebra $\mathcal A$. Assume that any externally definable subset of $M^2$ is measurable for the product algebra $\mathcal A^{\otimes 2}$. Let $\mu_0$ be a probability measure on $(M,\mathcal A)$. Then $\mu_0$ induces a global Keisler measure $\mu$ and $\mu$ is generically stable.
\end{thm}
\begin{proof}
Note first that by construction $\mu$ is finitely satisfiable in $M$ (hence also $M$-invariant).

We have at our disposal two different amalgams of $\mu_x$ by $\mu_y$. The first one is $\mu^{(2)}_{xy}=\mu_x \otimes \mu_y$ from the model-theoretic world. The second one comes from probability theory: we may form the product measure $\mu_0^{2}= \mu_0 \times \mu_0$ which is a $\sigma$-additive measure on $(M^2,\mathcal A^{\otimes 2})$. By hypothesis, $\mu_0^2$ induces a global Keisler measure $\mu_2$. Of course, $\mu^{(2)}$ and $\mu_2$ coincide on products $\phi(x)\wedge \psi(y)$ (they are both separated amalgam of $\mu_x$ and $\mu_y$). We will prove that in fact $\mu_2=\mu^{(2)}$. For this, it is enough to check that $\mu_2$ is an f.s.\,extension in $M$ of $\mu_x$.

Let $\theta(x,y), \phi(x) \in \mathcal L(\bar M)$. We have: 
\begin{align*}
\mu_0^2 (\theta(x,y) \wedge \phi(x)) &= \int_{(a,b) \in M^2} \theta(a,b) \wedge \phi(a) d\mu_0^{2}
\\ &= \int_{a \in \phi(M)} \int_{b \in M} \theta(a,b) d\mu_0 d\mu_0
\\ &\leq \mu_0(\phi(M)).\sup_{a \in \phi(M)} \int_{b \in M} \theta(a,b) d\mu_0
\\ &= \mu(\phi(x)).\sup_{a \in \phi(M)} \mu(\theta(a,y)).
\end{align*}

By Proposition \ref{fsunique}, this proves that $\mu_2=\mu_x \otimes \mu_y$.

By the usual Fubini theorem, $\mu_0^{\otimes 2}$ is a symmetric measure. Therefore it is also the case for $\mu^{(2)}$: $\mu^{(2)}(\theta(x,y))=\mu^{(2)}(\theta(y,x))$ for all $\theta(x,y) \in L(\bar M)$. By Theorem \ref{genstable}, this shows that $\mu$ is generically stable.
\end{proof}

\begin{rem}
The assumption that externally definable sets are measurable for the product $\sigma$-algebra is of course necessary. Consider for example an $\omega_1$-saturated model $M$ of RCF. Let $p\in S(M)$ be the type at $+\infty$ and $\tilde p$ the global co-heir of $p$. Then $\tilde p$ is induced by a $\sigma$-additive measure on $M$ (equipped with the Borel $\sigma$-algebra). Of course it is not generically stable, and note that the set $\{(x,y) \in M^2: x\leq y\}$ is not measurable for the product algebra.
\end{rem}

As a corollary, we recover the following result from \cite{NIP3}, Section 6.
\begin{cor}
Let $M$ be either $\mathbb R$: the standard real numbers equipped with any o-minimal structure  expanding the field operations, or $\mathbb Q_p$: the standard p-adic field. Let $\mu_0$ be a $\sigma$-additive measure on $M$, then $\mu_0$ induces a smooth Keisler measure $\mu$.
\end{cor}
\begin{proof}
Theorem \ref{sigma} implies that $\mu$ is generically stable, then using Corollary \ref{omin} or \ref{padic}, we deduce that it is smooth.
\end{proof}

\begin{question}
More generally, if we assume that $\mathcal A$ is generated as a $\sigma$-algebra by definable sets, is it the case that $\mu$ is smooth?
\end{question}

\begin{prop}
Let $M=\mathbb R$ be an o-minimal expansion of the standard model. Let $\mu$ be a smooth measure over $M^k$ concentrating on $[0,1]^k$. Then there is a $\sigma$-additive Borel measure $\mu_0$ on $[0,1]^k$ such that $\mu_0$ induces $\mu$ as a Keisler measure.
\end{prop}
\begin{proof}
Let $\Omega_k$ be the subspace of $S_k(M)$ of types that concentrate on $[0,1]^k$. Note that the standard part application $\st$ induces an application $\st: \Omega_k \rightarrow [0,1]^k$. This is a Borel map (the inverse image of a closed set is closed). In particular, we can consider the pushforward measure $\mu_0=\st_*(\mu)$. It is a $\sigma$-additive measure on $[0,1]^k$. All we now need to show is that $\mu_0$ induces $\mu$ as a Keisler measure, {\it i.e.}, that $\mu(X)=\mu_0(X)$ for any definable set $X$. 

Let $X$ be a definable set of $[0,1]^k$. Assume first that $X$ is closed. We have a definable map $d_X: M^k \rightarrow M$ such that $d_X(\bar x)$ is the distance of $\bar x$ to $X$. For $\epsilon>0$, let $X_\epsilon$ be the closed $\epsilon$-neighborhood of $X$: $X_\epsilon = \{\bar x \in [0,1]^k: d_X(\bar x) \leq \epsilon\}$. It is also a definable set. We have $\st^{-1}(X)=\cap_{\epsilon>0} X_\epsilon$. Let $p$ be a type in $\st^{-1}(X)\setminus X$. Then $(d_X)_*(p)$ is the type $0^+$ of $S(M)$. As $\mu$ is smooth, $(d_X)_*(\mu)$ is also smooth and it is not possible that it has an atom on $0^+$. Thus $\mu(\st^{-1}(X) \setminus X)=0$.

We treat the general case by induction on the dimension of $X$. The case of dimension 0 is trivial. So let $X$ be any definable set. Let $O$ be its interior and $\bar X$ its closure. Then $D=\bar X \setminus O$ has lower dimension then $X$. We know that $\mu(\bar X)=\mu_0(\bar X)$ and by induction $\mu(\bar X \setminus X)=\mu_0(\bar X \setminus X)$. Hence $\mu(X)=\mu_0(X)$.
\end{proof}

\end{document}